\documentclass[12pt]{amsart}
\usepackage{amssymb, amscd}

\begin{document}

\title[\null]
{Remarks on the extension of twisted Hodge metrics}

\author[\null]{Christophe Mourougane and Shigeharu TAKAYAMA} 

\thanks{A talk at the RIMS meeting 
``Bergman kernel and its applications to algebraic geometry'',
organized by T.\ Ohsawa, June 4 - 6, 2008.}

\maketitle
\baselineskip=18pt
\pagestyle{plain}


\theoremstyle{plain}
  \newtheorem{thm}{Theorem}[section]
  \newtheorem{main}[thm]{Main Theorem}
  \newtheorem{defthm}[thm]{Definition-Theorem}
  \newtheorem{prop}[thm]{Proposition}
  \newtheorem{lem}[thm]{Lemma}
  \newtheorem{cor}[thm]{Corollary}
  \newtheorem{conj}[thm]{Conjecture}
  \newtheorem{sublem}[thm]{Sublemma}
  \newtheorem{mainlem}[thm]{Main Lemma}
  \newtheorem{variant}[thm]{Variant}
\theoremstyle{definition}
  \newtheorem{dfn}[thm]{Definition}
  \newtheorem{exmp}[thm]{Example}
  \newtheorem{co-exmp}[thm]{Counter-Example}
  \newtheorem{prob}[thm]{Problem}
  \newtheorem{notation}[thm]{Notation}
  \newtheorem{quest}[thm]{Question}
  \newtheorem{setup}[thm]{Set up}
  \newtheorem{assum}[thm]{Assumption}
\theoremstyle{remark}
  \newtheorem{rem}[thm]{Remark}
  \newtheorem{com}[thm]{Comment}
\renewcommand{\theequation}{\thesection.\arabic{equation}}
\setcounter{equation}{0}

\newcommand{\BC}{{\mathbb{C}}}
\newcommand{\BN}{{\mathbb{N}}}
\newcommand{\BP}{{\mathbb{P}}}
\newcommand{\BQ}{{\mathbb{Q}}}
\newcommand{\BR}{{\mathbb{R}}}

\newcommand{\CalC}{{\mathcal{C}}}
\newcommand{\CalD}{{\mathcal{D}}}
\newcommand{\CE}{{\mathcal{E}}}
\newcommand{\CF}{{\mathcal{F}}}
\newcommand{\CH}{{\mathcal{H}}}
\newcommand{\CI}{{\mathcal{I}}}
\newcommand{\CJ}{{\mathcal{J}}}
\newcommand{\CK}{{\mathcal{K}}}
\newcommand{\CL}{{\mathcal{L}}}
\newcommand{\CM}{{\mathcal{M}}}
\newcommand{\CO}{{\mathcal{O}}}
\newcommand{\CS}{{\mathcal{S}}}
\newcommand{\CU}{{\mathcal{U}}}

\newcommand{\fm}{{\mathfrak{m}}}

\newcommand\ga{\alpha}
\newcommand\gb{\beta}
\def\th{\theta}
\newcommand\vth{\vartheta}
\newcommand\Th{\Theta}
\newcommand\ep{\varepsilon}
\newcommand\Ga{\Gamma}
\newcommand\w{\omega}
\newcommand\Om{\Omega}
\newcommand\Dl{\Delta}
\newcommand\del{\delta}
\newcommand\sg{\sigma}
\newcommand\vph{\varphi}
\newcommand\lam{\lambda}
\newcommand\Lam{\Lambda}

\newcommand\lra{\longrightarrow}
\newcommand\ra{\rightarrow}
\newcommand\ot{\otimes}
\newcommand\wed{\wedge}
\newcommand\ol{\overline}
\newcommand\isom{\, \wtil{\lra}\, }
\newcommand\sm{\setminus}

\newcommand\ai{\sqrt{-1}}
\newcommand\rd{{\partial}}
\newcommand\rdb{{\overline{\partial}}}
\newcommand\levi{\ai\rd\rdb}

\newcommand\wtil{\widetilde}
\newcommand\tf{\widetilde f}
\newcommand\tX{\widetilde X}
\newcommand\tx{\widetilde x}
\newcommand\tw{\widetilde \w}

\newcommand\what{\widehat}
\newcommand\hsg{\widehat \sg}

\newcommand\codim{\mbox{{\rm codim}}}
\newcommand\End{\mbox{{\rm End}}\, }
\newcommand\Exc{\mbox{{\rm Exc}}\, }
\newcommand\divisor{\mbox{{\rm div}}\, }
\newcommand\im{\mbox{{\rm im}}\, }
\newcommand\Ima{\mbox{{\rm Im}}\, }
\newcommand\Ker{\mbox{{\rm Ker}}\, }
\newcommand\Reg{\mbox{{\rm Reg}}\, }
\newcommand\Sing{\mbox{{\rm Sing}}\, }
\newcommand\Supp{\mbox{{\rm Supp}}\, }

\newcommand\Rq{R^qf_*(K_{X/Y} \ot E)}
\newcommand\F{R^qf_*(K_{X/Y} \ot E)}
\newcommand\FF{R^qf''_*(K_{X''/Y'} \ot E'')}

\newcommand\Xo{{X^\circ}}
\newcommand\Uo{{U^\circ}}
\newcommand\fo{{f^\circ}}
\newcommand\tauo{{\tau^\circ}}
\newcommand\go{{g_{\CO(1)}^\circ}}


\section{Introduction}

The aims of this text are to announce the result in a paper \cite{MT3}, 
to give proofs of some special cases of it, 
and to make comments and remarks for the proof given there.
Because the full proof in \cite{MT3} is much more involved and technical, 
we shall give a technical introduction and proofs for weaker statements 
in this text (see Theorem \ref{curve} and \ref{reduced}).
This text is basically independent from \cite{MT3}.

\subsection{Result in \cite{MT3}}

Our main concern is the positivity of direct image sheaves 
of adjoint bundles $\Rq$, for a K\"ahler morphism $f : X \lra Y$ endowed 
with a Nakano semi-positive holomorphic vector bundle $(E, h)$ on $X$.
In our previous paper \cite{MT2}, 
generalizing a result \cite{B} in case $q = 0$,
we obtained the Nakano semi-positivity of 
$\Rq$ with respect to the Hodge metric,
under the assumption that $f : X \lra Y$ is smooth.
However the smoothness assumption on $f$ is rather restrictive,
and it is desirable to remove it.

To state our result precisely, let us fix notations and recall basic facts.
Let $f : X \lra Y$ be a holomorphic map of complex manifolds.
A real $d$-closed $(1,1)$-form $\w$ on $X$ is said to be 
{\it a relative K\"ahler form} for $f$, if for every point $y \in Y$, 
there exists an open neighbourhood $W$ of $y$ and 
a smooth plurisubharmonic function $\psi$ on $W$ such that 
$\w + f^*(\ai\rd\rdb \psi)$ is a K\"ahler form on $f^{-1}(W)$.
A morphism $f$ is said to be {\it K\"ahler}, if there exists 
a relative K\"ahler form for $f$ (\cite[6.1]{Tk}),
and $f : X \lra Y$ is said to be a {\it K\"ahler fiber space},
if $f$ is proper, K\"ahler, and surjective with connected fibers,

\begin{setup} \label{basicMT}
(1)
Let $X$ and $Y$ be complex manifolds of $\dim X = n + m$ and 
$\dim Y = m$, and let $f : X \lra Y$ be a K\"ahler fiber space.
We do not fix a relative K\"ahler form for $f$, unless otherwise stated.
The {\it discriminant locus} of $f : X \lra Y$ is the minimum closed 
analytic subset $\Dl \subset Y$ such that $f$ is smooth over $Y \sm \Dl$.

(2)
Let $(E, h)$ be a Nakano semi-positive holomorphic vector bundle on $X$.
Let $q$ be an integer with $0 \le q \le n$.
By Koll\'ar \cite{Ko1} and Takegoshi \cite{Tk}, 
$\F$ is torsion free on $Y$,
and moreover it is locally free on $Y \sm \Dl$ (\cite[4.9]{MT2}).
In particular we can let $S_q \subset \Dl$ be 
the minimum closed analytic subset of $\codim_Y S_q \ge 2$ such that 
$\F$ is locally free on $Y \sm S_q$.
Let $\pi : \BP(\F|_{Y \sm S_q}) \lra Y \sm S_q$
be the projective space bundle, and let $\pi^*(\F|_{Y \sm S_q}) \lra \CO(1)$
be the universal quotient line bundle.

(3)
Let $\w_f$ be a relative K\"ahler form for $f$.
Then we have the Hodge metric $g$ on the vector bundle 
$\F|_{Y \sm \Dl}$ with respect to $\w_f$ and $h$ (\cite[\S 5.1]{MT2}).
By the quotient 
$\pi^*(\Rq|_{Y \sm \Dl}) \lra \CO(1)|_{\pi^{-1}(Y \sm \Dl)}$,
the metric $\pi^*g$ gives the quotient metric $\go$ on 
$\CO(1)|_{\pi^{-1}(Y \sm \Dl)}$.
The Nakano, even weaker Griffiths, semi-positivity of $g$
(by \cite[1.2]{B} for $q=0$, and by \cite[1.1]{MT2} for $q$ general) 
implies that $\go$ has a semi-positive curvature.
\qed
\end{setup}

In these notations, the main result in \cite{MT3} is as follows.

\begin{thm} \label{MT}
Let $f : X \lra Y$, $(E, h)$ and $0 \le q \le n$ be as in Set up \ref{basicMT}.

(1) Unpolarized case.
Then, for every relatively compact open subset $Y_0 \subset Y$,
the line bundle $\CO(1)|_{\pi^{-1}(Y_0 \sm S_q)}$ on 
$\BP(\Rq|_{Y_0 \sm S_q})$ 
has a singular Hermitian metric 
with semi-positive curvature, and which is smooth on $\pi^{-1}(Y_0 \sm \Dl)$.

(2) Polarized case.
Let $\w_f$ be a relative K\"ahler form for $f$.
Assume that there exists a closed analytic set 
$Z \subset \Dl$ of $\codim_Y Z \ge 2$ such that 
$f^{-1}(\Dl)|_{X \sm f^{-1}(Z)}$ is a divisor
and has a simple normal crossing support (or empty).
Then the Hermitian metric $\go$ on $\CO(1)|_{\pi^{-1}(Y \sm \Dl)}$
can be extended as a singular 
Hermitian metric $g_{\CO(1)}$ with semi-positive curvature
of $\CO(1)$ on $\BP(\Rq|_{Y \sm S_q})$. 
\end{thm}

Theorem \ref{MT}\,(1) is reduced to Theorem \ref{MT}\,(2) for
$f' = f \circ \mu : X' \lra Y$ after a modification $\mu : X' \lra X$.
Then however the induced map $f' : X' \lra Y$ is only locally K\"ahler 
in general.
Hence we need to restrict everything on relatively compact subsets
of $Y$ in Theorem \ref{MT}\,(1).

If in particular in Theorem \ref{MT}, $\Rq$ is locally free and 
$Y$ is a smooth projective variety, then the vector bundle 
$\Rq$ is pseudo-effective in the sense of \cite[\S 6]{DPS}.
This notion \cite[\S 6]{DPS} is a natural generalization of the fact that
on a smooth projective variety, a divisor $D$ is pseudo-effective
(i.e., a limit of effective divisors)
if and only if the associated line bundle $\CO(D)$ admits a singular 
Hermitian metric with semi-positive curvature. 
The above curvature property of $\CO(1)$ leads to the following
algebraic positivity of $\Rq$.

\begin{thm} \label{algebraic}
Let $f : X \lra Y$ be a surjective morphism with connected fibers
between smooth projective varieties, and let $(E, h)$ be 
a Nakano semi-positive holomorphic vector bundle on $X$.
Then the torsion free sheaf $\Rq$ is weakly positive over $Y \sm \Dl$
(the smooth locus of $f$), in the sense of Viehweg \cite[2.13]{Vi2}.
\end{thm}

See \cite[\S1]{MT3} for further introduction.

\subsection{Statement in this text} \label{this}

We would like to explain the proofs of the following two theorems
in this text.
Because there is no essential limitations of the number of pages,
we may repeat some arguments and make comments repetitiously.

\begin{setup} \label{basic}
(General set up.) \ 
Let $f : X \lra Y$ be a holomorphic map of complex manifolds,
which is proper, K\"ahler, surjective with connected fibers,
and $f$ is smooth over the complement $Y \sm \Dl$ of 
a closed analytic subset $\Dl \subset Y$. 
Let $\w_f$ be a relative K\"ahler form for $f$, and let $(E, h)$ 
be a Nakano semi-positive holomorphic vector bundle on $X$.
Let $q$ be a non-negative integer.

It is known by Koll\'ar \cite{Ko1} and Takegoshi \cite{Tk} that 
$\F$ is torsion free,
and moreover it is locally free where $f$ is smooth (\cite[4.9]{MT2}).
In particular we can let $S_q \subset \Dl$ be 
the minimum closed analytic subset of $\codim_Y S_q \ge 2$ such that 
$\F$ is locally free on $Y \sm S_q$.
Once we take a relative K\"ahler form $\w_f$ for $f$, we then have
the Hodge metric $g$ on the vector bundle $\F|_{Y \sm \Dl}$ 
with respect to $\w_f$ and $h$ (\cite[\S 5.1]{MT2} or Remark \ref{defHodge}).
\qed
\end{setup}

\begin{thm} \label{curve}
In Set up \ref{basic}, assume further that $\dim Y = 1$.
Let $L$ be a quotient holomorphic line bundle of $\Rq$.
Then $L$ has a singular Hermitian metric with semi-positive curvature,
whose restriction on $Y \sm \Dl$ is the quotient metric of 
the Hodge metric $g$ on $\Rq|_{Y \sm \Dl}$.
\end{thm}

\begin{thm} \label{reduced}
In Set up \ref{basic}, assume further that 
$f$ has reduced fibers in codimension 1 on $Y$, 
i.e., there exists a closed analytic set $Z \subset \Dl$ of 
$\codim_Y Z \ge 2$ such that every fiber of $y \in Y \sm Z$ is reduced.
Let $L$ be a holomorphic line bundle on $Y$ with a surjection
$\Rq|_{Y \sm Z} \lra L|_{Y \sm Z}$.
Then $L$ has a singular Hermitian metric with semi-positive curvature,
whose restriction on $Y \sm \Dl$ is the quotient metric of 
the Hodge metric $g$ on $\Rq|_{Y \sm \Dl}$.
\end{thm}


The above assumptions:\ 
$\dim Y = 1$, and/or with reduced fibers, 
or even fibers are semi-stable,
are quite usual in algebraic geometry.
In this sense, the assumptions in Theorem \ref{curve} and \ref{reduced}
are not so artificial.


\subsection{Complement}

Here is a comment on the relation between the statements in \S 1.1
and those in \S 1.2.
Although we will not give proofs, we can pursue the method of proof
of Theorem \ref{curve} and \ref{reduced} to show the following two 
statements, as we show Theorem \ref{MT} in \cite{MT3}.

\begin{thm} \label{curve+}
In Set up \ref{basic}, assume further that $\dim Y = 1$.
Then the line bundle $\CO(1)$ for 
$\pi : \BP(\Rq) \lra Y$ has a singular Hermitian metric $g_{\CO(1)}$ 
with semi-positive curvature,
and whose restriction on $\pi^{-1}(Y \sm \Dl)$ is the quotient metric 
$\go$ of $\pi^*g$,
where $g$ is the Hodge metric with respect to $\w_f$ and $h$.
\end{thm}

\begin{thm} \label{reduced+}
In Set up \ref{basic}, assume further that 
$f$ has reduced fibers in codimension 1 on $Y$.
Then the line bundle $\CO(1)$ for 
$\pi : \BP(\Rq|_{Y \sm S_q}) \lra Y \sm S_q$ has a singular 
Hermitian metric $g_{\CO(1)}$ with semi-positive curvature,
and whose restriction on $\pi^{-1}(Y \sm \Dl)$ is the quotient metric 
$\go$ of $\pi^*g$, where $g$ is the Hodge metric 
with respect to $\w_f$ and $h$.
\end{thm}

One clear difference between \S 1.1 and \S 1.2 is
geometric conditions on $f : X \lra Y$.
Another is about line bundles to be considered, namely $\CO(1)$ or $L$.
For example, 
Theorem \ref{curve+} (or \ref{MT}) concerns all rank 1 quotient of $\Rq$,
while Theorem \ref{curve} concerns a rank 1 quotient of $\Rq$,
hence Theorem \ref{curve+} is naturally stronger than Theorem \ref{curve}.
In fact Theorem \ref{curve+} implies Theorem \ref{curve} by a standard
argument (\cite[\S 6.2]{MT3}).
The proof of Theorem \ref{curve+} (as well as Theorem \ref{MT})
requires another uniform estimate which does not depend on
rank 1 quotients $L$ of $\Rq$, other than the uniform estimate
given in Lemma \ref{Ft11} of the proof of Theorem \ref{curve}.


\section{Preliminary Arguments}

\subsection{Localization}

As the next lemma shows, to see our theorems,
we can neglect codimension 2 analytic subsets of $Y$.

\begin{lem} \label{codim2}
Let $Y$ be a complex manifold, and $Z$ a closed analytic subset
of $Y$ with $\codim_Y Z \ge 2$.
Let $L$ be a holomorphic line bundle on $Y$ with a singular Hermitian
metric $h$ on $L|_{Y \sm Z}$ with semi-positive curvature.
Then $h$ extends as a singular Hermitian metric on $L$ with
semi-positive curvature.
\end{lem}

\begin{proof}
Let $W$ be a small open subset of $Y$ with a nowhere vanishing
section $e \in H^0(W, L)$.
Then a function $h(e, e)$ on $W \sm Z$ can be written as
$h(e, e) = e^{-\vph}$ with a plurisubharmonic function 
$\vph$ on $W \sm Z$.
By Hartogs type extension for plurisubharmonic functions,
$\vph$ can be extended uniquely as a plurisubhamonic function 
$\wtil \vph$ on $W$.
Then $e^{-\wtil \vph}$ gives the desired extension of $h$ on $W$.
\end{proof}

In particular, we can neglect the set $S_q$ (resp.\ $Z$) 
in Set up \ref{basic} (resp.\ in Theorem \ref{reduced}), 
and only consider codimension 1 part of the discriminant locus $\Dl$.
Once we obtain the Hodge metric $g$ of $\F|_{Y \sm \Dl}$
or the quotient metric $g^\circ_L$ of $L|_{Y \sm \Dl}$,
the extension property of $g^\circ_L$ is a local question.
Hence we can further reduce our situation to the following

\begin{setup} \label{local}
(Generic local set up.) \
Let $Y$ be (a complex manifold which is biholomorphic to) 
a unit ball in $\BC^m$ with coordinates $t = (t_1, \ldots, t_m)$,
$X$ a complex manifold of $\dim X = n+m$ with a K\"ahler form $\w$.
Let $f : X \lra Y$ be a proper surjective holomorphic map 
with connected fibers.
Let $(E, h)$ be a Nakano semi-positive holomorphic vector bundle on $X$,
and let $q$ be an integer with $0 \le q \le n$.
Let $K_Y \cong \CO_Y$ be a trivialization by a nowhere vanishing
section $dt = dt_1 \wed \ldots \wed dt_m \in H^0(Y, K_Y)$.
Let $g$ be the Hodge metric on $\Rq|_{Y \sm \Dl}$ with respect to 
$\w$ and $h$.
Let us assume the following:\

(1) $f$ is flat, and the discriminant locus $\Dl \subset Y$ is 
$\Dl = \{t_m = 0\}$.

(2) $\Rq \cong \CO_Y^{\oplus r}$, i.e.,
globally free and trivialized of rank $r$.

(3) 
Let $f^*\Dl = \sum b_iB_i$ be the prime decomposition.
For every $B_i$, the induced morphism $f : \Reg B_i \lra \Dl$ 
is surjective and smooth.
Here $\Reg B_i$ is the smooth locus of $B_i$.
If $B_i \ne B_j$, the intersection $B_i \cap B_j$ does not contain
any fiber of $f$.

We may replace $Y$ by slightly smaller balls, or 
may assume everything is defined over a slightly larger ball. 
\qed
\end{setup}

\begin{rem}
(1) For this moment, in Set up \ref{local},
we do not assume that $\dim Y = 1$, nor that $f$ has reduced fibers.

(2) Set up \ref{local}\,(3) is automatically satisfied in case $\dim Y = 1$.

(3) Refer \cite[5.2]{MT2} for the replacement of 
a relative K\"ahler form $\w_f$ by a K\"ahler form $\w$.
\qed
\end{rem}

\begin{notation}
(1) For a non-negative integer $d$, we set $c_d = \ai^{d^2}$.

(2)
Let $f : X \lra Y$ be as in Set up \ref{local}.
We set $\Omega_{X/Y}^p = \bigwedge^p \Om_{X/Y}^1$
rather formally, because we will only deal 
$\Omega_{X/Y}^p$ on which $f$ is smooth.
For an open subset $U \subset X$ where $f$ is smooth,
and for a differentiable form $\sg \in A^{p,0}(U, E)$, 
we say $\sg$ is relatively holomorphic
and write $[\sg] \in H^0(U, \Omega_{X/Y}^p \ot E)$,
if $\sg \wed f^*dt \in H^0(U, \Omega_X^{p+m} \ot E)$. 
\qed
\end{notation}

\subsection{Relative hard Lefschetz type theorem}

We discuss in Set up \ref{local}.

One fundamental ingredient, even in the definition of Hodge metrics,
is the following proposition.
In case $q = 0$, this is quite elementary.

\begin{prop} \label{Takegoshi}
\cite[5.2]{Tk}. 
There exist $H^0(Y, \CO_Y)$-module homomorphisms
\begin{equation*} 
\begin{aligned}
	* \circ \CH &: H^0(Y, \F) 
		\lra H^0(X, \Om_X^{n+m-q} \ot E), \\
	L^q &: H^0(X, \Om_X^{n+m-q} \ot E) 
		\lra H^0(Y, \F) 
\end{aligned}
\end{equation*} 
such that
(1) $(c_{n+m-q}/q!) L^q \circ (* \circ \CH) = id$, and
(2)
for every $u \in H^0(Y, \F)$, there exists a relative
holomorphic form
$[\sg_u] \in H^0(X \sm f^{-1}(\Dl), \Om_{X/Y}^{n-q} \ot E)$
such that
$$
	(* \circ \CH (u))|_{X \sm f^{-1}(\Dl)} = \sg_u \wed f^*dt.
$$ 
\end{prop}

\begin{proof}
We take a smooth strictly plurisubhamonic exhaustion function $\psi$
on $Y$, for example $\|t\|^2$.
Recalling $\F = K_{Y}^{\ot (-1)} \ot R^qf_*(K_X \ot E)$,
the trivialization $K_{Y} \cong \CO_Y$ by $dt$ gives an isomorphism
$\F \cong R^qf_*(K_X \ot E)$.
Since $Y$ is Stein, we have also a natural isomorphism
$H^0(Y, R^qf_*(K_X \ot E)) \cong H^q(X, K_X \ot E)$.
We denote by $\ga^q$ the composed isomorphism
$$
 \ga^q : H^0(Y, \F) 
	\isom H^q(X, K_X \ot E).
$$

With respect to the K\"ahler form $\w$ on $X$, 
we denote by $*$ the Hodge $*$-operator, and by 
$$
	L^q : H^0(X, \Om_X^{n+m-q} \ot E) 
			\lra H^q(X, K_X \ot E)
$$ 
the Lefschetz homomorphism induced from ${\w}^q \wed \bullet$. 
Also with respect to $\w$ and $h$, we set
$\CH^{n+m,q}(X, E, f^*\psi) 
= \{u \in A^{n+m, q}(X, E) ;\ \rdb u = \vth_h u = 0, \
	e(\rdb (f^*\psi))^* u = 0 \}$.
(We do not explain what this space of harmonic forms is,
because the definition is not important in this text.) \ 
By \cite[5.2.i]{Tk}, $\CH^{n+m,q}(X, E, f^*\psi)$ represents 
$H^q(X, K_X \ot E)$ as an $H^0(Y, \CO_Y)$-module, 
and hence there exists a natural isomorphism
$$
\iota : \CH^{n+m,q}(X, E, f^*\psi) 
		\isom H^q(X, K_X \ot E)
$$
given by taking the Dolbeault cohomology class.
We have an isomorphism
$$
	\CH  = \iota^{-1} \circ \ga^q :  
     H^0(Y, \F) \isom \CH^{n+m,q}(X, E, f^*\psi).
$$
Also by \cite[5.2.i]{Tk},
the Hodge $*$-operator gives an injective homomorphism 
$$
   * : \CH^{n+m,q}(X, E, f^*\psi)
	\lra H^0(X, \Om_X^{n+m-q} \ot E),
$$
and induces a splitting 
$* \circ \iota^{-1} : H^q(X, K_X \ot E)
	\lra H^0(X, \Om_X^{n+m-q} \ot E)$
for the Lefschetz homomorphism $L^q$ such that
$(c_{n+m-q}/q!) L^q \circ * \circ \iota^{-1} = id$.
(The homomorphism $\delta^q$ in \cite[5.2.i]{Tk} with respect to 
$\w$ and $h$ is $* \circ \iota^{-1}$ times a universal constant.) \
In particular 
$$
	(c_{n+m-q}/q!) ((\ga^q)^{-1} \circ L^q) \circ (* \circ \CH ) 
	= id.
$$
All homomorphisms $\ga^q, *, L^q, \iota, \CH$ are as 
$H^0(Y, \CO_Y)$-modules.

Let $u \in H^0(Y, \F)$. 
Then we have $* \circ \CH (u) \in H^0(X, \Om_X^{n+m-q} \ot E)$,
and then by \cite[5.2.ii]{Tk}
$$
	(* \circ \CH (u))|_{X \sm f^{-1}(\Dl)} 
	= \sg_u \wed f^*dt
$$ 
for some $[\sg_u] \in H^0(X \sm f^{-1}(\Dl), \Om_{X/Y}^{n-q} \ot E)$.
It is not difficult to see 
$[\sg_u] \in H^0(X \sm f^{-1}(\Dl), \Om_{X/Y}^{n-q} \ot E)$
does not depend on the particular choice of a global frame 
$dt$ of $K_{Y}$.
\end{proof}

\begin{rem} \label{defHodge}
We recall the definition of the Hodge metric $g$ of $\F|_{Y \sm \Dl}$ 
with respect to $\w$ and $h$ \cite[5.1]{MT2}.
We only mention it for a global section $u \in H^0(Y, \F)$. 
It is given by
$$
 g(u,u)(t) 
	= \int_{X_{t}} (c_{n-q}/q!) 
	  ({\w}^q \wed \sg_u \wed h \ol{\sg_u})|_{X_{t}}
$$
at $t \in Y \sm \Dl$.
In the relation 
$$
	(* \circ \CH (u))|_{X \sm f^{-1}(\Dl)} = \sg_u \wed f^*dt,
$$
the left hand side is holomorphically extendable across $f^{-1}(\Dl)$,
and is non-vanishing if $u$ is, in an appropriate sense.
In the right hand side, $f^*dt$ may only have zero along $f^{-1}(\Dl)$,
that is ``Jacobian'' of $f$, and hence $\sg_u$ may only have ``pole''
along $f^{-1}(\Dl)$.
This is the main reason why $g(u,u)(t)$ has a positive
lower bound on $Y \sm \Dl$, and which is fundamental
for the extension of positivity (see (5) of the proof of 
Proposition \ref{weak} below).
The importance of the role of the Jacobian of $f$ is already
observed by Fujita \cite{Ft}.
\qed
\end{rem}

\subsection{Non-uniform estimate}

Here we state a weak extension property.
This is a basic reason for all extension of positivity
of direct image sheaves of relative canonical bundles,
for example in \cite{Ft}, \cite{Ka1}, \cite{Vi1}, and so on.
However this is not enough to conclude the results in \S 1.

\begin{prop} \label{weak}
In Set up \ref{basic}, let $W \subset Y$ be an open subset, and let 
$u \in H^0(W \sm S_q, \Rq)$ which is nowhere vanishing on $W \sm S_q$.
Then the smooth plurisubharmonic function $-\log g(u,u)$ on $W \sm \Dl$
can be extended as a plurisubharmonic function on $W$.
\end{prop}

\begin{proof}
We may assume $W = Y$.
Moreover it is enough to consider in Set up \ref{local} as before.
In particular $S_q = \emptyset$ and $\Dl = \{t_m=0\}$.
We shall discuss the extension property at the origin $t = 0 \in Y$,
and hence we replace $Y$ by a small ball centered at $t=0$.

(1)
By Proposition \ref{Takegoshi},
we have $* \circ \CH (u) \in H^0(X, \Om_X^{n+m-q} \ot E)$.
This $* \circ \CH (u)$ does not vanish identically along 
$\Dl = \{t_m = 0 \} \subset Y$ as an element of
$H^0(Y, \CO_Y)$-module $H^0(X, \Om_X^{n+m-q} \ot E)$.
This is saying that there exists at least one component $B_j$
in $f^*\Dl = \sum b_iB_i$ such that $* \circ \CH (u)$ does not vanish 
of order greater than or equal to $b_j$ along $B_j$.
We take one such $B_j$ and denote by 
$$
	B = B_j \text{ and } b = b_j.
$$

(2) 
We take a general point $x_0 \in B \cap f^{-1}(0)$ so that $x_0$ is a smooth
point on $(f^*\Dl)_{red}$, and take local coordinates
$(U; z = (z_1, \ldots, z_{n+m}))$ centered at $x_0 \in X$.
We may assume $f(U) = Y$ and 
$t = f(z) = (z_{n+1}, \ldots, z_{n+m-1}, z_{n+m}^b)$ on $U$.

Over $U$, the bundle $E$ is also trivialized, i.e.,
$E|_U \cong U \times \BC^{r(E)}$, where $r(E)$ is the rank of $E$.
Using the local trivializations on $U$, we have a constant
$a > 0$ such that
(i) $\w \ge a \w_{eu}$ on $U$, 
where $\w_{eu} = \ai/2\sum_{i=1}^{n+m} dz_i \wed d\ol{z_i}$
is the standard complex euclidean K\"ahler form,
and
(ii) $h \ge a \text{Id}$ on $U$ as Hemitian matrixes.
Here we regard $h|_U(x)$ as a positive definite Hermitian matrix 
at each $x \in U$ in terms of $E|_U \cong U \times \BC^{r(E)}$, 
and here $\text{Id}$ is the $r(E) \times r(E)$ identity matrix.

(3)
By Proposition \ref{Takegoshi}, we can write as
$(* \circ \CH (u))|_{X \sm f^{-1}(\Dl)} = \sg_u \wed f^*dt$ for some
$\sg_u \in A^{n-q,0}(X \sm f^{-1}(\Dl), E)$.
We write
$\sg_u = \sum_{I \in I_{n-q}} \sg_{I} dz_I + R$ on $U \sm B$.
Here $I_{n-q}$ is the set of all multi-indexes 
$1 \le i_1 < \ldots < i_{n-q} \le n$ of length $n-q$ 
(not including $n+1, \ldots, n+m$),
$\sg_{I} = {}^t(\sg_{I,1}, \ldots, \sg_{I,r(E)})$ is a vector valued 
holomorphic function with $\sg_{I,i} \in H^0(U \sm B, \CO_X)$,
and here $R = \sum_{k=1}^m R_{k} \wed dz_{n+k} \in A^{n-q,0}(U \sm B, E)$. 
Now
$$
	\sg_u \wed f^*dt 
	= b z_{n+m}^{b-1} \bigg(\sum_{I \in I_{n-q}} \sg_{I} dz_I \bigg) 
		\wed dz_{n+1} \wed \ldots \wed dz_{n+m}
$$
on $U \sm B$.
Since $\sg_u \wed f^*dt = (* \circ \CH (u))|_{X \sm f^{-1}(\Dl)}$ and 
$* \circ \CH (u) \in H^0(X, \Om_X^{n+m-q} \ot E)$,
all $z_{n+m}^{b-1} \sg_{I}$ can be extended holomorphically on $U$.
By the non-vanishing property of $* \circ \CH (u)$ along $bB$,
we have at least one $\sg_{J_0,i_0} \in H^0(U \sm B, \CO_X)$ 
whose divisor is 
$$
	\divisor (\sg_{J_0,i_0}) = - pB|_U+D
$$ 
with some integer $0 \le p \le b-1$, and 
an effective divisor $D$ on $U$ not containing $B|_U$.
We take such 
$$
	J_0 \in I_{n-q} \text{ and } i_0 \in \{1, \ldots, r(E)\}.
$$
(Now $\divisor (\sg_{J_0,i_0}) = - pB|_U+D$ is fixed.) \ 
We set 
$$
	Z_u = \{ y \in \Dl ; \ D \text{ contains } B|_U \cap f^{-1}(y) \}.
$$
We can see that $Z_u$ is not Zariski dense in $\Dl$, because
otherwise $D$ contains $B|_U$, and also that $Z_u$ is Zariski closed
of $\codim_Y Z_u \ge 2$ (particularly using $f$ is flat).

(4)
We take any point $y_1 \in \Dl \sm Z_u$, and 
a point $x_1 \in B|_U \cap f^{-1}(y_1)$ such that $x_1 \not\in D$.
Let $0 < \ep \ll 1$ be a sufficiently small number so that,
on the $\ep$-polydisc neighbourhood
$U(x_1, \ep) = \{z = (z_1, \ldots, z_{n+m}) \in U ; \ 
|z_i - z_i(x_1)| < \ep \text{ for any } 1 \le i \le n+m \}$,
we have
$$
	A := \inf\{ |\sg_{J_0,i_0}(z)| ; \ z \in U(x_1, \ep) \sm B \} > 0.
$$
We should note that $\sg_{J_0,i_0}$ may have a pole along $B$,
but no zeros on $U(x_1, \ep)$.
We set $Y' := f(U(x_1, \ep))$ which is an open neighbourhood of
$y_1 \in Y$, since $f$ is flat (in particular it is an open mapping).
Then for any $t \in Y' \sm \Dl$, we have 
\begin{equation*} 
\begin{aligned}
\int_{X_{t}}(c_{n-q}/q!) 
(\w^q \wed \sg_u \wed h\ol{\sg_u})|_{X_{t}} 
&
\ge
a \int_{X_{t} \cap U}(c_{n-q}/q!) 
	(\w^q \wed \sg_u \wed \ol{\sg_u})|_{X_{t} \cap U} \\
& = 
a^{q+1} \int_{z \in X_{t} \cap U}
	\sum_{I \in I_{n-q}} \sum_{i=1}^r |\sg_{I,i}(z)|^2 dV_n \\
& \ge 
a^{q+1} \int_{z \in X_{t} \cap U(x_1, \ep)} A^2 \ dV_n \\
& = a^{q+1} A^2 (\pi \ep^2)^n.
\end{aligned}
\end{equation*} 
Here $dV_n = (\ai/2)^n \bigwedge_{i=1}^n dz_i \wed d\ol{z_i}$
is the standard euclidean volume form in $\BC^n$.
Namely we have $g(u, u)(t) \ge a^{q+1} A^2 (\pi \ep^2)^n$
for any $t \in Y' \sm \Dl$.

(5)
We proved that $-\log g(u,u)$ is bounded from above
around every point of $\Dl \sm Z_u$.
This means that a plurisubharmonic function $-\log g(u,u)$ on $Y \sm \Dl$ 
can be extended as a plurisubharmonic function on $Y \sm Z_u$ 
by Riemann type extension, and hence as 
a plurisubharmonic function on $Y$ by Hartogs type extension.
\end{proof}

\begin{rem}
Here are some remarks when we try to generalize the proof above
to obtain Theorem \ref{curve} and \ref{reduced}.
The point is the set $Z_u$ above depends on $u$.
This is the main difficulty when we consider an extension property of 
quotient metrics.
In that case, we need to obtain a uniform estimate of $g(u_s, u_s)$
for a family $\{u_s\}$.
If $s$ moves, then $Z_{u_s}$ also may move and cover a larger subset
of $\Dl$, which may not be negligible for the extension of 
plurisubharmonic functions.

The intersection $B|_U \cap D$ is a set of indeterminacies.
If (a part of) a fiber $f^{-1}(y)$ is contained in $B|_U \cap D$,
the analysis of the behavior of $g(u, u)$ around such $y$
is quite hard and in fact indeterminate.
This is why we do not want to touch $Z_u$. 
In some geometric setting as below, we can avoid such phenomena.
We can delete one of two in the right hand side of
$\divisor (\sg_{I,i}) = -pB|_U + D$.

(i)
In case $\dim Y = 1$, we can take $D = 0$.
This is because, if a prime divisor $\Gamma$ on $U$ contains
$B|_U \cap f^{-1}(y)$, then $\Gamma = B|_U$.
In case when $\dim Y = 1$, $q = 0$ and $E = \CO_X$, 
a uniform estimate is cleared by Fujita \cite[1.11]{Ft}
(as we will see below).
This will lead Theorem \ref{curve}.

(ii) 
In case the fibers of $f$ are reduced, we can take $p = 0$
(cf.\ $0 \le p \le b-1$ in (3) of the proof above).
This will lead Theorem \ref{reduced}.

To deal with a general case in \cite{MT3}, we use a semi-stable 
reduction for $f$. 
A computation of Hodge metrics is a kind of an estimation of integrals,
which usually can be done only after a good choice of local coordinates.
A semi-stable reduction can be seen as a resolution of singularities
of a map $f : X \lra Y$.
Then the crucial point is to compair two Hodge metrics:\
the original one and the one after taking a semi-stable reduction.
\qed
\end{rem}



\section{Proof of Theorems}

\subsection{Quotient metric} \label{quot}

We discuss in Set up \ref{local}.

We denote by $F = \Rq$ which is locally free on $Y$,
and by $r$ the rank of $F$.
We have a smooth Hermitian metric $g$ defined on $Y \sm \Dl$ (not on $Y$).
Let $F \lra L$ be a quotient line bundle with the kernel $M$:\
$0 \lra M \lra F \lra L \lra 0$ (exact).
We first describe the quotient metric on $L|_{Y \sm \Dl}$.
We take a frame $e_1, \ldots, e_r \in H^0(Y, F)$ over $Y$
such that $e_1, \ldots, e_{r-1}$ generate $M$.
Then the image 
$$
	\what e_r \in H^0(Y, L)
$$ 
of $e_r$ under $F \lra L$ generates $L$.
We represent the Hodge metric $g$ on $Y \sm \Dl$ in terms of this frame
as $g_{i\ol j} = g(e_i, e_j) \in A^0(Y \sm \Dl, \BC)$.
At each point $t \in Y \sm \Dl$, $(g_{i\ol j}(t))_{1 \le i, j \le r}$
is a positive definite Hermitian matrix, in particular,
$(g_{i\ol j}(t))_{1 \le i, j \le r-1}$ is also positive definite.
We let $(g^{\ol i j}(t))_{1 \le i, j \le r-1}$ be the inverse matrix.
Then the pointwise orthogonal projection of $e_r$ to 
$(M|_{Y \sm \Dl})^\perp$ with respect to $g$ is given by
$$
  P(e_r) 
  = e_r - \sum_{i=1}^{r-1}\sum_{j=1}^{r-1}e_ig^{\ol i j}g_{j \ol r}
  \ \in A^0(Y \sm \Dl, F).
$$
We have in fact $P(e_r) - e_r \in A^0(Y \sm \Dl, M)$ and 
$g(P(e_r), s) = 0$ for any $s \in A^0(Y \sm \Dl, M)$.
Then the quotient metric on $L|_{Y \sm \Dl}$
is defined by
$$
	g^\circ_L(\what e_r, \what e_r) = g(P(e_r), P(e_r)) 
$$
on $Y \sm \Dl$.

It is well-known after Griffiths that the curvature 
does not decrease by a quotient.
In our setting, the Nakano semi-positivity of $(F|_{Y \sm \Dl}, g)$
\cite[1.1]{MT2}, or even weaker the Griffiths semi-positivity 
implies that $(L|_{Y \sm \Dl}, g^\circ_L)$ is semi-positive.
In particular if we write $g^\circ_L(\what{e_r}, \what{e_r}) = e^{-\vph}$
with $\vph \in A^0(Y \sm \Dl, \BR)$, this $\vph$ is plurisubharmonic 
on $Y \sm \Dl$.
If we can show $\vph$ is extended as a plurisubharmonic function on $Y$,
then $g^\circ_L$ extends as a singular Hermitian metric on $L$ over $Y$
with semi-positive curvature.
By virtue of Riemann type extension for plurisubharmonic functions,
it is enough to show that $\vph$ is bounded from above
(i.e., $g^\circ_L(\what e_r, \what e_r)$ is bounded from below 
by a positive constant) around every point $y \in \Dl$.
In the next two subsections, we shall prove the following

\begin{lem} \label{Ft13+} 
In Set up \ref{local} and the notations above, assume further that
$\dim Y = 1$, or that $f$ has reduced fibers.
Let $y \in \Dl$.
Then there exists a neighbourhood $Y'$ of $y \in Y$ and a positive number
$N$ such that $g^\circ_L(\what e_r, \what e_r)(t) \ge N$ for any 
$t \in Y' \sm \Dl$.
\end{lem}

\begin{cor}
Theorem \ref{curve} and Theorem \ref{reduced} hold true.
\end{cor}

We introduce the following notations for the following arguments.
For $s = (s_1, \ldots, s_r) \in \BC^r$, we let
$u_s = \sum_{i=1}^r s_ie_i \in H^0(Y, F)$.
We note that $u_s$ is nowhere vanishing on $Y$ as soon as $s \ne 0$.
We also note that, with respect to the standard topology of $\BC^r$ and 
the topology of uniform convergence on compact sets for 
$H^0(X, \Om_X^{n+m-q} \ot E)$,
the map $\BC^r \lra H^0(X, \Om_X^{n+m-q} \ot E)$ given by 
$s \mapsto u_s \mapsto * \circ \CH (u_s) 
= \sum_{i=1}^r s_i (* \circ \CH (e_i))$, 
is continuous.
Let $S^{2r-1} = \{ s \in \BC^r ; \ |s| = (\sum |s_i|^2)^{1/2} = 1 \}$.

\subsection{Over curves}

We shall prove Theorem \ref{curve} by showing Lemma \ref{Ft13+} in this case.
It is enough to consider in Set up \ref{local} with $\dim Y =1$.
In particular $Y = \{ t \in \BC ; \ |t| < 1 \}$ a unit disc, and
$\Dl = 0 \in Y$ the origin.
We will use both $\Dl \subset Y$ and $t=0 \in Y$ to compair
our argument here with a general case.
Let $F = \F \lra L$ be a quotient line bundle, and use the same notation
in \S \ref{quot}, in particular we have a frame 
$e_1, \ldots, e_r \in H^0(Y, F)$, $\what e_r \in H^0(Y, L)$
generates $L$ and so on.
We use $u_s = \sum_{i=1}^r s_ie_i \in H^0(Y, F)$ for 
$s = (s_1, \ldots, s_r) \in \BC^r$.
The key is to obtain the following uniform bound.

\begin{lem}   \label{Ft11}
(cf.\ \cite[1.11]{Ft}.) \ 
In Set up \ref{local} with $\dim Y =1$ and the notation above,
let $s_0 \in S^{2r-1}$.
Then there exist a neighbourhood $S(s_0)$ of $s_0$ in $S^{2r-1}$,
a neighbourhood $Y'$ of $0 \in Y$ and a positive number $N$ such that
$g(u_s, u_s)(t) \ge N$ for any $s \in S(s_0)$ and any $t \in Y' \sm \Dl$.
\end{lem}

\begin{proof} 
We denote by $f^*\Dl = \sum b_iB_i$.

(1) 
By Proposition \ref{Takegoshi}, we have
$* \circ \CH (u_{s_0}) \in H^0(X, \Om_X^{n+1-q} \ot E)$.
This $* \circ \CH (u_{s_0})$ does not vanish at $t = 0$ as an element 
of $H^0(Y, \CO_Y)$-module $H^0(X, \Om_X^{n+1-q} \ot E)$.
Then there exists a component 
$B_j$ in $f^*\Dl = \sum b_iB_i$ such that 
$(* \circ \CH)(u_{s_0})$ does not vanish 
of order greater than or equal to $b_j$ along $B_j$.
We take one such $B_j$ and denote by $B = B_j \text{ and } b = b_j$.

(2) 
We take a general point $x_0 \in B$ so that 
$x_0$ is a smooth point on $(f^*\Dl)_{red} = f^{-1}(0)$, 
and take local coordinates $(U; z = (z_1, \ldots, z_{n+1}))$ 
centered at $x_0 \in X$ such that $t = f(z) = z_{n+1}^b$ on $U$.
Over $U$, the bundle $E$ is also trivialized.
Using the local trivializations on $U$, we have a constant
$a > 0$ such that
(i) $\w \ge a \w_{eu}$ on $U$, 
where $\w_{eu} = \ai/2\sum_{i=1}^{n+1} dz_i \wed d\ol{z_i}$,
and
(ii) $h \ge a \text{Id}$ on $U$ as Hemitian matrixes,
as in the proof of Proposition \ref{weak}.

(3)
Let $s \in S^{2r-1}$.
By Proposition \ref{Takegoshi}, we can write as
$(* \circ \CH (u_s))|_{X \sm f^{-1}(\Dl)} = \sg_s \wed f^*dt$ for some
$\sg_s \in A^{n-q,0}(X \sm f^{-1}(\Dl), E)$.
We write
$\sg_s = \sum_{I \in I_{n-q}} \sg_{sI} dz_I + R_s \wed dz_{n+1}$ 
on $U \sm B$.
Here $I_{n-q}$ is the set of all multi-indexes 
$1 \le i_1 < \ldots < i_{n-q} \le n$ of length $n-q$, 
$\sg_{sI} = {}^t(\sg_{sI,1}, \ldots, \sg_{sI,r(E)})$ with
$\sg_{sI,i} \in H^0(U \sm B, \CO_X)$, 
and here $R_s \wed dz_{n+1} \in A^{n-q,0}(U \sm B, E)$. 
Now
$$
	\sg_s \wed f^*dt 
	= b z_{n+1}^{b-1} \bigg(\sum_{I \in I_{n-q}} \sg_{sI} dz_I \bigg) 
		 \wed dz_{n+1}
$$ 
on $U \sm B$.
Since $\sg_s \wed f^*dt = (* \circ \CH(u_s))|_{X \sm f^{-1}(\Dl)}$ and 
$* \circ \CH(u_s) \in H^0(X, \Om_X^{n+1-q} \ot E)$,
all $z_{n+1}^{b-1} \sg_{sI}$ can be extended holomorphically on $U$.

At the point $s_0 \in S^{2r-1}$, 
by the non-vanishing property of $* \circ \CH(u_{s_0})$ along $bB$,
we have at least one $\sg_{s_0J_0,i_0} \in H^0(U \sm B, \CO_X)$ 
whose divisor is $\divisor (\sg_{s_0J_0,i_0}) = - p_0B|_U$ 
with some integer $0 \le p_0 \le b-1$ (being $x_0 \in B|_U$ 
general, and $U$ sufficiently small).
Here we used $\dim Y = 1$.
We take such $J_0 \in I_{n-q}$ and $i_0 \in \{1, \ldots, r(E)\}$.
By the continuity of $s \mapsto u_s \mapsto * \circ \CH(u_s)$, 
we can take the same $J_0$ and $i_0$ for any $s \in S^{2r-1}$ near $s_0$, 
so that $\divisor (\sg_{sJ_0,i_0}) = - p(s)B|_U$ with 
the order $p(s)$ satisfies $0 \le p(s) \le p_0 = p(s_0)$
for any $s \in S^{2r-1}$ near $s_0$.

(4)
By the continuity of $s \mapsto u_s \mapsto * \circ \CH(u_s)$, 
we can take an $\ep$-polydisc neighbourhood
$U(x_0, \ep) = \{z = (z_1, \ldots, z_{n+1}) \in U ; \ 
|z_i - z_i(x_0)| < \ep \text{ for any } 1 \le i \le n+1 \}$ for some
$\ep > 0$, and a neighbourhood $S(s_0)$ of $s_0$ in $S^{2r-1}$ such that
$A := \inf\{ |\sg_{sJ_0,i_0}(z)| ; \
	 s \in S(s_0), \ z \in U(x_0, \ep) \sm B \} > 0$.
We should note that $\sg_{s J_0, i_0}$ may have a pole along $B$,
but no zeros on $U(x_0, \ep)$.
We set $Y' := f(U(x_0, \ep))$ which is an open neighbourhood of
$0 \in Y$, since $f$ is flat.
Then for any $s \in S(s_0)$ and any $t \in Y' \sm \Dl$, we have
$g(u_s, u_s)(t) \ge a^{q+1} A^2 (\pi \ep^2)^n$ as in Proposition \ref{weak}.
\end{proof}

\begin{lem}  \label{Ft12}
(cf.\ \cite[1.12]{Ft}.) \
There exist a neighbourhood $Y'$ of $0 \in Y$ and a positive number
$N$ such that $g(u_s, u_s)(t) \ge N$ for any $s \in S^{2r-1}$
and any $t \in Y' \sm \Dl$.
\end{lem}

\begin{proof}
Since $S^{2r-1}$ is compact, this is clear from Lemma \ref{Ft11}.
\end{proof}

\begin{lem} \label{Ft13}
(cf.\ \cite[1.13]{Ft}.) \  
There exists a neighbourhood $Y'$ of $0 \in Y$ and a positive number
$N$ such that $g^\circ_L(\what e_r, \what e_r)(t) \ge N$ for any $t \in Y' \sm \Dl$.
\end{lem}


\begin{proof}
We take a neighbourhood $Y'$ of $0 \in Y$ and a positive number $N$
in Lemma \ref{Ft12}.
We may assume $Y'$ is relatively compact in $Y$.
We put $s_i = - \sum_{j=1}^{r-1} g^{\ol i j}g_{j \ol r}
\in A^0(Y \sm \Dl, \BC)$ for $1 \le i \le r-1$, and $s_r = 1$.
Then $P(e_r) = \sum_{i=1}^r s_i e_i$ on $Y \sm \Dl$.
For every $t \in Y' \sm \Dl$, we have
$s = (s_1, s_2, \ldots, s_r) \in \BC^r \sm \{0\}$, and 
$s(t)/|s(t)| \in S^{2r-1}$.
Then for any $t \in Y' \sm \Dl$, we have
$g^\circ_L(\what e_r, \what e_r)(t) = g(u_{s(t)}, u_{s(t)})(t)
= |s(t)|^2 g(u_{s(t)/|s(t)|}, u_{s(t)/|s(t)|})(t) \ge N$,
since $s/|s| \in S^{2r-1}$.
\end{proof}


\subsection{Fiber reduced}

We shall prove Theorem \ref{reduced} by the same strategy 
in the previous subsection.
By Lemma \ref{codim2} we may assume the set $Z$ 
in Theorem \ref{reduced} is empty.
It is enough to consider in Set up \ref{local} with $f^*\Dl = \sum B_i$.
Let $F = \F \lra L$ be a quotient line bundle, and use the same notation
in \S \ref{quot}, in particular we have a frame 
$e_1, \ldots, e_r \in H^0(Y, F)$, $\what e_r \in H^0(Y, L)$
generates $L$ and so on.
We use $u_s = \sum_{i=1}^r s_ie_i \in H^0(Y, F)$ for 
$s = (s_1, \ldots, s_r) \in \BC^r$.
As we observed in the previous subsection, it is enough to show
the following

\begin{lem} 
(cf.\ \cite[1.11]{Ft}.) \  
In Set up \ref{local} and the notation above, let $s_0 \in S^{2r-1}$.
Then there exist a neighbourhood $S(s_0)$ of $s_0$ in $S^{2r-1}$,
a neighbourhood $Y'$ of $0 \in Y$ and a positive number $N$ such that
$g(u_s, u_s)(t) \ge N$ for any $s \in S(s_0)$ and any $t \in Y' \sm \Dl$.
\end{lem}

\begin{proof} 
(1) 
By Proposition \ref{Takegoshi}, we have
$* \circ \CH(u_{s_0}) \in H^0(X, \Om_X^{n+m-q} \ot E)$. 
This $* \circ \CH(u_{s_0})$ does not vanish at $t = 0$ as an element 
of $H^0(Y, \CO_Y)$-module $H^0(X, \Om_X^{n+m-q} \ot E)$.
There exists a component $B_j$ in $f^*\Dl = \sum B_i$
such that $* \circ \CH(u_{s_0})$ does not vanish identically along 
$B_j \cap f^{-1}(0)$.
Here we used our assumption in Theorem \ref{reduced} that 
$f$ has reduced fibers.
In fact, if $* \circ \CH(u_{s_0})$ does vanish identically along 
all $B_i \cap f^{-1}(0)$ in $f^*\Dl = \sum B_i$, then 
$* \circ \CH(u_{s_0})$ vanishes at $t = 0$ as an element 
of $H^0(Y, \CO_Y)$-module $H^0(X, \Om_X^{n+m-q} \ot E)$,
and leads a contradiction.
We take one such $B_j$ and denote by $B = B_j$ (with $b = b_j = 1$).

(2) 
We take a point $x_0 \in B \cap f^{-1}(0)$ such that
$* \circ \CH(u_{s_0})$ does not vanish at $x_0$,
and that $f^*\Dl$ is smooth at $x_0$.
We then take local coordinates
$(U; z = (z_1, \ldots, z_{n+m}))$ centered at $x_0 \in X$
such that $t = f(z) = (z_{n+1}, \ldots, z_{n+m-1}, z_{n+m})$ on $U$.
Over $U$, the bundle $E$ is also trivialized.
Using the local trivializations on $U$, we have a constant
$a > 0$ such that
(i) $\w \ge a \w_{eu}$ on $U$, 
where $\w_{eu} = \ai/2\sum_{i=1}^{n+m} dz_i \wed d\ol{z_i}$,
and
(ii) $h \ge a \text{Id}$ on $U$ as Hemitian matrixes,
as in the proof of Proposition \ref{weak}.

(3)
Let $s \in S^{2r-1}$.
By Proposition \ref{Takegoshi}, we can write as
$(* \circ \CH(u_s))|_{X \sm f^{-1}(\Dl)} = \sg_s \wed f^*dt$ for some
$\sg_s \in A^{n-q,0}(X \sm f^{-1}(\Dl), E)$.
We write
$\sg_s = \sum_{I \in I_{n-q}} \sg_{sI} dz_I + R_s$ on $U \sm B$.
Here $I_{n-q}$ is the set of all multi-indexes 
$1 \le i_1 < \ldots < i_{n-q} \le n$ of length $n-q$, 
$\sg_{sI} = {}^t(\sg_{sI,1}, \ldots, \sg_{sI,r(E)})$ with
$\sg_{sI,i} \in H^0(U \sm B, \CO_X)$, 
and here $R_s = \sum_{k=1}^m R_{sk} \wed dz_{n+k} \in A^{n-q,0}(U \sm B, E)$. 
Now
$$
	\sg_s \wed f^*dt 
	= \bigg(\sum_{I \in I_{n-q}} \sg_{sI} dz_I \bigg) 
		 \wed dz_{n+1} \wed \ldots \wed dz_{n+m}
$$ 
on $U \sm B$.
Since $\sg_s \wed f^*dt = (* \circ \CH(u_s))|_{X \sm f^{-1}(\Dl)}$ and 
$* \circ \CH(u_s) \in H^0(X, \Om_X^{n+m-q} \ot E)$,
all $\sg_{sI}$ can be extended holomorphically on $U$.

At the point $s_0 \in S^{2r-1}$, 
by the non-vanishing property of $* \circ \CH(u_{s_0})$ at $x_0$,
we have at least one $\sg_{s_0J_0,i_0} \in H^0(U \sm B, \CO_X)$ 
whose divisor is $\divisor (\sg_{s_0J_0,i_0}) = D_0$ with some
effective divisor $D_0$ on $U$ not containing $x_0$.
This is because, if all $\sg_{s_0I,i}$ vanish at $x_0$,
we see $* \circ \CH(u_{s_0}) = \sg_{s_0} \wed f^*dt$ (now on $U$) 
vanishes at $x_0$, and we have a contradiction.
We take such $J_0 \in I_{n-q}$ and $i_0 \in \{1, \ldots, r(E)\}$.
By the continuity of $s \mapsto u_s \mapsto * \circ \CH(u_s)$, 
we can take the same $J_0$ and $i_0$ for any $s \in S^{2r-1}$ near $s_0$.
By the same token, the divisor $D(s)$ may depend on $s \in S^{2r-1}$, but
we can keep the condition that $D(s)$ does not contain 
$B|_U \cap f^{-1}(0)$ if $s \in S^{2r-1}$ is close to $s_0$.

(4)
Then by the continuity of $s \mapsto u_s \mapsto * \circ \CH(u_s)$, 
we can take an $\ep$-polydisc neighbourhood
$U(x_0, \ep) = \{z = (z_1, \ldots, z_{n+m}) \in U ; \ 
|z_i - z_i(x_0)| < \ep \text{ for any } 1 \le i \le n+m \}$ for some
$\ep > 0$, and a neighbourhood $S(s_0)$ of $s_0$ in $S^{2r-1}$ such that
$A := \inf\{ |\sg_{sJ_0,i_0}(z)| ; \
	 s \in S(s_0), \ z \in U(x_0, \ep) \sm B \} > 0$.
We set $Y' := f(U(x_0, \ep))$ which is an open neighbourhood of
$0 \in Y$, since $f$ is flat.
Then for any $s \in S(s_0)$ and any $t \in Y' \sm \Dl$, we have
$g(u_s, u_s)(t) \ge a^{q+1} A^2 (\pi \ep^2)^n$ as in 
Proposition \ref{weak}.
\end{proof}



\section{Examples}			

\newcommand{\pb}{\mathbb{P}}\newcommand{\cb}{\mathbb{C}}
\newcommand{\zb}{\mathbb{Z}}\newcommand{\rb}{\mathbb{R}}
\newcommand{\oc}{\mathcal{O}}

Here are some related examples and counter-examples of the positivity of 
direct image sheaves, including cases the total space 
$X$ can be singular.
These are due to Wi\'sniewski and H\"oring, and taken from \cite{Ho}.

Our general set up is as follows.
We take a vector bundle $E$ of rank $n+2$ over a smooth projective
variety $Y$.
Denote by $p : \pb(E)\lra Y$ the natural (smooth) projection.
We take a hypersurface $X$ in $\pb(E)$ cut out by a section of
$N:=\oc_E(d)\ot p^* \lambda$ for $d>0$ and some line bundle 
$\lambda$ on $Y$.
Denote by $f : X \lra Y$ the induced (non necessary smooth) map 
of relative dimension $n$.
Because $X$ is a divisor, the sheaf 
$\w_{P(E)/Y} \ot N \ot \oc_X$ equals $\w_{X/Y}$.
We choose a line bundle $L:=\oc_E(k)\ot p^* \mu$ with $k>0$ and with
a line bundle $\mu$ on $Y$, and set $L_X := L|_X$ the resrtiction on $X$.
We then consider the exact sequence
$$
	0 \lra \w_{\pb(E)/Y} \ot L \lra \w_{\pb(E)/Y}\ot N\ot L
	  \lra \w_{X/Y}\ot L_X \lra 0.
$$
Note that since $L$ is $p$-ample, we have $R^1 p_* (\w_{\pb(E)/Y}\ot L)=0$.
We push the sequence forward by $p$ to get the following exact sequence 
of sheaves on $Y$:
$$
	0\lra p_*(\w_{\pb(E)/Y}\ot L)
	\lra p_*(\w_{\pb(E)/Y}\ot N\ot L)
	\lra f_* (\w_{X/Y}\ot L_X) \lra 0.
$$
Remember that $\w_{\pb (E)/Y}=\oc_E(-n-2) \ot p^*\det E$, so that 
$p_* (\w_{\pb(E)/Y}\ot \oc_E(k))=0$ for $k< n+2$, and that
$p_* (\w_{\pb(E)/Y}\ot \oc_E(k))=S^{k-n-2}E\ot\det E$ 
for $k\geq n+2$.

\begin{exmp} (\cite[2.C]{Ho}.) \
Choose $Y=\pb^1$, $E = \oc_{\pb^1}(-1)^{\oplus 2}\oplus \oc_{\pb^1}$, 
$N = \oc_E(2)$ that is effective and defines $X$,
and $L = \oc_E(1)\ot p^* \oc_{\pb^1}(1)$ that is semi-positive.
The push-forward sequence reads
$$
	0\lra 0\lra \oc_{\pb^1}(-1)\lra f_* (\w_{X/Y}\ot L_X) 
	\lra 0.
$$
Hence $f_* (\w_{X/Y}\ot L_X)$ is negative. 
The point is that here, $X$ is not reduced.
\qed
\end{exmp}

\begin{exmp} (\cite[2.D]{Ho}.) \ 
Choose $Y=\pb^1$ and $E = \oc_{\pb^1}(-1)\oplus \oc_{\pb^1}^{\oplus 3}$. 
Take $N = \oc_E(4)$ whose generic section defines $X$. 
This scheme is a $3$-fold smooth outside the $1$-dimensional base locus 
$\pb (\oc_{\pb^1}(-1))\subset\pb (E)$, Gorenstein as a divisor, 
and normal since smooth in codimension 1. 
Choose $L=\oc_E(k)\ot p^* \oc_{\pb^1}(k)$ that is semi-positive.
The push-forward sequence shows that for $1\leq k< 4$, 
$S^k E\ot \oc_{\pb^1}(k-1) = f_* (\w_{X/Y}\ot L_X)$ 
is not nef.
For $k \geq 4$, the push-forward sequence reads
$$
	0\lra S^{k-4}E\ot \oc_{\pb^1}(k-1) 
	\overset{\sg}{\lra} 
	S^k E\ot \oc_{\pb^1}(k-1)\lra f_* (\w_{X/Y}\ot L_X)
	\lra 0.
$$
Here the map $\sigma$ is given by the contraction with the section 
$s\in H^0(\pb(E), N)$ $=$ $H^0(Y,S^4 E)$ $=$ 
$H^0(Y,S^4\oc_{\pb^1}^{\oplus 3})$, 
whereas the quotient 
$S^k E/\text{Im} (S^4\oc_{\pb^1}^{\oplus 3}\ot S^{k-4}E)$ 
contains the factor $\oc_{\pb^1}(-1)^{\oplus k}$. 
Hence $f_* (\w_{X/Y}\ot L_X)$ is not weakly positive.
The point here is that the locus of non-rational singularities of 
$X$ projects onto $Y$ by $f : X \lra Y$.
\qed
\end{exmp}

\begin{exmp} (\cite[2.A]{Ho}.) \ 
Choose $Y$ to be $\pi : Y = \pb (F) \lra \BP^3$, where 
$F:=\oc_{\pb^3}(2)^{\oplus 2}\oplus \oc_{\pb^3}$ is semi-ample but not ample.
Choose $E$ to be 
$\oc_F(1)^{\oplus 2} \oplus (\oc_F(1) \ot \pi^* \oc_{\pb^3}(1))$.
Wi\'sniewski showed that the linear system 
$|N|:=|\oc_E(2) \ot p^*\pi^*\oc_{\pb^3}(-2)|$
has a smooth member, that we denote by $X$.
Remark that $L:=\oc_E(1)$ is semi-positive, but 
the push-forward sequence shows that
$$
	\oc_F(3) \ot \pi^*\oc_{\pb^3}(-1) = f_* (\w_{X/Y}\ot L_X)
$$
is not nef.
The point here is that the conic bundle $f : X \lra Y$ 
has some non-reduced fibers,
that make the direct image only weakly positive.
\qed
\end{exmp}




\baselineskip=14pt

\vskip 10pt

Christophe Mourougane

\vskip 5pt

Institut de Recherche Math\'ematique de Rennes

Campus de Beaulieu

35042 Rennes cedex, France

e-mail: christophe.mourougane@univ-rennes1.fr

\vskip 10pt

Shigeharu Takayama

\vskip 5pt

Graduate School of Mathematical Sciences

University of Tokyo

3-8-1 Komaba, Tokyo

153-8914, Japan

e-mail: taka@ms.u-tokyo.ac.jp

\end{document}